\documentclass[12pt]{amsart}
\usepackage{amssymb}
\usepackage{latexsym}
\usepackage[pdftex]{graphicx}
\usepackage{dsfont}

\newtheorem{theorem}{Theorem}
\newtheorem{lemma}{Lemma}
\newtheorem{proposition}{Proposition}

\newtheorem{corollary}{Corollary}

\newtheorem{remark}{Remark}

\newcommand{\al}{\alpha}

\newcommand{\om}{\omega}
\newcommand{\Om}{\Omega}

\def\R{{\mathbb R}}

\newcommand{\wt}{\widetilde}

\newcommand{\beq}{\begin{equation}}
\newcommand{\eeq}{\end{equation}}
\newcommand{\beqna}{\begin{eqnarray*}}
\newcommand{\eeqna}{\end{eqnarray*}}
\newcommand{\beqn}{\begin{equation*}}
\newcommand{\eeqn}{\end{equation*}}
\newcommand{\bp}{\begin{proof}}
\newcommand{\ep}{\end{proof}}
\newcommand{\bprop}{\begin{proposition}}
\newcommand{\eprop}{\end{proposition}}
\newcommand{\bt}{\begin{theorem}}
\newcommand{\et}{\end{theorem}}
\newcommand{\bex}{\begin{Example}}
\newcommand{\eex}{\end{Example}}
\newcommand{\bc}{\begin{corollary}}
\newcommand{\ec}{\end{corollary}}
\newcommand{\bl}{\begin{lemma}}
\newcommand{\el}{\end{lemma}}

\renewcommand{\ge}{\geqslant}
\renewcommand{\le}{\leqslant}

\begin{document}

\title[Positivity probability under the spectral gap condition]{On the probability that a stationary Gaussian process with spectral gap remains non-negative on a long interval}

\author[N. Feldheim, O. Feldheim, B. Jaye, F. Nazarov, S. Nitzan]{Naomi Feldheim, Ohad Feldheim, Benjamin Jaye, Fedor Nazarov, and Shahaf Nitzan}

\address{Naomi Feldheim, Department of Mathematics, Weizmann Institute of Science, Rehovot 76100, Israel}
\email{naomi.feldheim@weizmann.ac.il}
\address{Ohad Feldheim, The Hebrew university of Jerusalem, Jerusalem 9190401, Israel}
\email{ohad.feldheim@mail.huji.ac.il}
\address{Benjamin Jaye, Department of Mathematical Sciences, Clemson University, Clemson, SC 29631, USA}
\email{bjaye@clemson.edu}
\address{Fedor Nazarov, Department of Mathematical Sciences, Kent State University, Kent, OH 44242, USA}
\email{nazarov@math.kent.edu}
\address{Shahaf Nitzan, School of Mathematics, Georgia Institute of Technology, Atlanta GA 30332, USA}  
\email{shahaf.nitzan@math.gatech.edu}



\thanks{The authors are  
supported in part by the Science Foundation of the
Israel Academy of Sciences and Humanities grant 166/11,
the ERC advanced grant, and the U.S.~National Science Foundation grants DMS-1600726, DMS-1600239, 
DMS-1500881 respectively}

 \begin{abstract}
Let $f$ be a zero mean continuous stationary Gaussian process on ${\mathbb R}$ whose spectral measure vanishes in 
a $\delta$-neighborhood of the origin. Then the probability that $f$ stays non-negative on an interval 
of length $L$ is at most $e^{-c\delta^2L^2}$ with some absolute $c>0$ and the result is sharp 
without additional assumptions.
\end{abstract}

\maketitle

\section{Introduction}

Let $f$ be a zero mean continuous stationary Gaussian process on $\mathbb R$ and let $L>0$. We are interested
in good bounds for the probability
 that $f$ stays non-negative on some fixed interval
of length $L$ (since $f$ is stationary, this probability does not depend on the location of 
the interval, so we can always assume that our interval is just $[0,L]$). It has been recently observed in 
\cite{FF}, \cite{KK} and \cite{FFN} that in many interesting cases, one can get reasonably sharp bounds for this 
probability from both above and below in terms of the behavior of the spectral measure $\mu$ of the 
Gaussian process $f$ near the origin, usually under the assumption that $\mu$ has a non-trivial absolutely
continuous component. In the present paper, we will prove the following
\begin{theorem}
\label{spectralgap}
Let $f$ be a zero mean continuous stationary Gaussian process on $\mathbb R$ whose spectral measure $\mu$ has a gap,
i.e., 
$$
\mu([-\delta,\delta])=0\quad\text{for some }\delta>0\,.
$$ 
Then 
\begin{equation}
\mathcal P\{f\ge 0\text{ on }[0,L]\}\le e^{-c\delta^2L^2}
\label{main}
\end{equation}
with some absolute constant $c>0$. 
\end{theorem}

We will show in Section \ref{sharpness} that this bound cannot be improved in general. In particular, one can 
prove a matching bound from below for any zero mean stationary Gaussian process whose spectral measure is 
compactly supported and has a non-trivial absolutely continuous component,
extending the results of \cite{KK} where it was done for a class of processes with discrete time. 

An analogue of Theorem \ref{spectralgap} (with the same proof) holds for stationary Gaussian processes 
on $\mathbb Z$.

The inequality (\ref{main}) itself is not new: it appeared in \cite{FFN} already. The novelty here is
that the gap condition is the only one we impose on the spectral measure $\mu$; no other a priori
assumptions of any kind are made about it. This seems to put our setup beyond the scope of the techniques 
used in \cite{FFN} despite the fact that our present argument follows the same approach.

\section{Basic facts about stationary Gaussian processes}

A Gaussian process on $\R$ is a mapping $f:\R\times\Om\to \R$ (where $(\Om,\mathcal F,\mathcal P)$
is a probability space) that is jointly measurable and satisfies the condition that for 
every $x_1,\dots,x_m\in\R$, the vector $(f(x_1,\cdot),\dots,f(x_m,\cdot))$ has a (possibly degenerate)
Gaussian distribution in $\R^m$. We shall always assume that our Gaussian processes have mean zero, i.e.,
$\mathcal E[f(x)]=0$ for all $x\in\R$. We say that $f$ is continuous if $f(\cdot,\om)$ is a continuous 
function on $\R$ almost surely (i.e., for $\mathcal P$-almost every $\om\in\Om$). We say that $f$ is 
stationary if for every $x_1,\dots,x_m\in\R$, the distribution of the vector 
$(f(x_1+x,\cdot),\dots,f(x_m+x,\cdot))$ does not depend on $x\in\R$. In what follows, we will often 
suppress the probability variable $\omega$ in the notation $f(x,\omega)$ and just write $f(x)$ instead.
Also, we will always assume that our process $f$ is not identically $0$.

If a Gaussian process $f$ is continuous and stationary, then its covariance kernel 
$K(x,y)=\mathcal E[f(x)f(y)]$ can be written as $K(x,y)=k(x-y)$ with some positive definite 
(in the sense that $\sum_{i,j=1}^m k(x_i-x_j)c_jc_j\ge 0$ for all $x_i\in\R$ and all real numbers $c_i$) continuous function
$k:\R\to\R$. By the Bochner theorem, there exists a non-negative symmetric with respect to the
origin finite measure $\mu$ on $\R$ such that 
$$
k(x)=\widehat\mu(x)=\int_{\R}e^{2\pi i xy}\,d\mu(y)\,.
$$
This measure $\mu$ is called the spectral measure of $f$. Conversely, given any finite symmetric measure $\mu$ on $\R$ that decays not too slowly near infinity, we can 
construct a unique continuous stationary Gaussian process $f$ on $\R$ whose spectral measure is $\mu$.

If $f_j$ are independent continuous stationary Gaussian processes with spectral measures $\mu_j$ 
and the series $\sum_j f_j$ converges uniformly on compact subsets of $\R$ almost surely, then
its sum is a continuous stationary Gaussian process whose spectral measure is $\mu=\sum_j \mu_j$. Conversely,
if $f$ is a continuous stationary Gaussian process with spectral measure $\mu$ and $\mu$ is 
represented as a countable sum $\mu=\sum_j\mu_j$, then, under some mild extra assumptions, $f$ can
be decomposed into a uniformly converging on compact subsets of $\R$ sum of independent
continuous stationary Gaussian processes $f_j$ with spectral measures $\mu_j$.

We refer the reader who wants to learn more about continuous and smooth Gaussian processes to 
the Appendix in \cite{NS} and books and articles mentioned therein.

The words ``decaying not too slowly'' and ``under some mild extra assumptions'' in the above two paragraphs 
can be given precise meaning by stating the corresponding (nearly) optimal assumptions explicitly.
However, the following remark shows that for the purpose of proving Theorem \ref{spectralgap}, we do not need the full 
strength of the corresponding delicate theory, but can get away with very crude sufficient conditions instead. 

\begin{remark}\rm
Let $f$ be any continuous stationary Gaussian process with some spectral measure $\mu$ 
(possibly decaying very slowly). 
Then, for every non-negative compactly supported smooth mollifier $\eta:\R\to \R$, the convolution
$f*\eta$ is a continuous (and even smooth) stationary Gaussian process whose spectral measure
is $|\widehat\eta|^2\mu$. 

If $\eta$ is supported on $[-\varepsilon,0]$, then the condition
$f\ge 0$ on $[0,L]$ implies that $f*\eta\ge 0$ on $[0,L-\varepsilon]$, so
$$
\mathcal P\{f\ge 0\text{ on }[0,L]\}\le\mathcal P\{f*\eta\ge 0\text{ on }[0,L-\varepsilon]\}\,. 
$$
\end{remark}
Since 
$$
(|\widehat\eta|^2\mu)(\mathbb R\setminus[-R,R])=\int_{\mathbb R\setminus[-R,R]}|\widehat\eta|^2\,d\mu
\le \mu(\R)\max_{\mathbb R\setminus[-R,R]}|\widehat\eta|^2\,,
$$
we can thus reduce the general case to the case when the spectral measure $\mu$ of the process
has the property that $\mu(\R\setminus[-R,R])$ decays faster than any power of $R$ as $R\to\infty$,
in which case all our claims about convergence, decompositions, existence, etc. in the course of the proof
of Theorem \ref{spectralgap}
become totally routine. 
 
Finally a few words should be said about stationary 
Gaussian processes with compactly supported spectral measures.
In this case the Gaussian process represents a random entire function $f(z)$ on the whole complex plane, its
derivatives are also (zero mean) stationary jointly Gaussian processes on $\mathbb R$, and we can pass to the limit
in the usual spectral measure identity
$$
\mathcal E\Bigl[\Bigl|\sum_{j}c_j f(x_j)\Bigr|^2\Bigr]=
\int_{\mathbb R}\Bigl|\sum_k c_j e^{2\pi i x_jy}\Bigr|^2\,d\mu(y)\,,\quad c_j\in\mathbb C
$$
for appropriate difference ratios to get
$$
\mathcal E\Bigl[\Bigl|\sum_{k=0}^n a_k f^{(k)}(0)\Bigr|^2\Bigr]=
\int_{\mathbb R}\Bigl|\sum_{k=0}^n a_k(2\pi i y)^k\Bigr|^2\,d\mu(y)\,,\quad a_k\in\mathbb C.
$$
The classical reference for all these facts is \cite{B}.

\section{The idea of the proof}

The key steps of our approach are easiest to explain in the discrete setting when we have a stationary 
Gaussian process $f$ on $\mathbb Z$ instead of $\mathbb R$ and the task is to estimate the probability
$\mathcal P\{f(0),f(1),\dots,f(L)\ge 0\}$. In this case the corresponding spectral measure $\mu$ lives 
on the unit circle $\mathbb T$, which can be identified with the interval $[-\frac 12,\frac 12]$ by the
standard mapping $[-\frac 12,\frac 12]\ni t\mapsto e^{2\pi it}\in\mathbb T$. The spectral gap condition
means that $\mu(\mathcal L_\delta)=0$ for some arc $\mathcal L_\delta=\{e^{2\pi it}\,:\,|t|\le\delta\}$.

The general idea is just to create a linear combination $g=\sum_{k=0}^L\beta_k f(k)$ with positive
coefficients $\beta_k$ summing to $1$ such that the quantity
$$
\mathcal E[g^2]=\int_{\mathbb T}\Bigl|\sum_{k=0}^L\beta_k z^k\Bigr|^2\,d\mu(z)
$$
is small while many of the coefficients $\beta_k$ are relatively
large. Then we can say that, in the event of positivity
we are interested in, many of the Gaussian random variables $f(k)$ should be small simultaneously. If we could in addition show that the covariance matrix of some large subset of them is sufficiently non-degenerate, 
we would be able to conclude that the overall positivity probability is rather small.

This scheme is usually implemented in two separate steps. First one chooses a polynomial 
$P(z)=\sum_{k=0}^L\beta_k z^k$ with $\beta_k\ge 0, \sum_{k=0}^L\beta_k=1$ 
such that $\int_{\mathbb T}|P|^2\,d\mu$ is small 
under the spectral gap condition. A simple universal example of such a polynomial is just
$P_{L,\text{simple}}(z)=\left(\frac{1+z}{2}\right)^L$, which gives 
$\int_{\mathbb T}|P|^2\,d\mu\le e^{-2cL}\mu(\mathbb T)$ with some $c=c(\delta)>0$.
It is slightly inconvenient to work with $P_{L,\text{simple}}$ directly 
because its coefficients decay rather fast when 
going away from the middle. This drawback can be easily fixed by considering 
$$
P_{L,\text{universal}}(z)=\frac{1+z+\dots+z^{2L'}}{2L'+1}P_{L',\text{simple}}(z)
$$ 
with $L'\approx \frac L3$ instead. This new polynomial is also uniformly exponentially small on 
$\mathbb T\setminus\mathcal L_{\delta}$, so we still have $\mathcal E[g^2]\le e^{-2cL}\mu(\mathbb T)$, 
while it has $L'+1$ coefficients in the middle equal to $\frac 1{2L'+1}\ge \frac 1L$.

Next one tries to find a subset $\Lambda\subset\{0,1,\dots, L\}$ for which the corresponding
coefficients $\beta_k$ are not too small and the covariance matrix of the family 
$\{f(k)\,:\,k\in\Lambda\}$ is reasonably non-degenerate. For instance, if we could find a subset $\Lambda$ 
of cardinality $|\Lambda|\ge cL$ near the middle of $\{0,1,\dots,L\}$ 
such that the corresponding $f(k)$ are nearly independent (i.e., their 
covariance matrix has a bounded condition number), then we would write something like
\begin{multline}
\label{naive}
\mathcal P\{f(0),f(1),\dots,f(L)\ge 0\}
\\
\le 
\mathcal P\{g\ge e^{-\frac c2L}\sqrt{\mu(\mathbb T)}\}
+\mathcal P\{f(k)\le \beta_k^{-1}e^{-\frac c2L}\sqrt{\mu(\mathbb T)}\text{ for every }k\in\Lambda\}
\\
\le \exp\left(-\frac 12e^{cL}\right)+\left(CLe^{-\frac c2 L}\right)^{|\Lambda|}\
\le e^{-c'L^2}
\end{multline}
and be done.

This second step can, indeed, be carried out as said for spectral
measures $\mu$ with non-trivial absolutely continuous
component but it fails miserably if the support of $\mu$ is thin, in which case large covariance matrices
can be as degenerate as one wants. However it is rather clear that under such circumstances our universal
polynomial $P_{L,\text{universal}}$ does a poor job of minimizing $\int_{\mathbb T}|P|^2\,d\mu$
and we should better replace it by something depending on $\mu$.

This is exactly what we are doing below. We {\em start} with looking at how degenerate the 
covariance matrix $A$ of $f(0),f(1),\dots,f(\ell)$ may be for a given $\ell$ (which the reader should
think of as a small multiple of $L$). Note that, due to the stationarity, any $\ell+1$ consecutive
values of $f$ will have the same covariance matrix. 

The least eigenvalue $\sigma^2$ of $A$ is given by
$$
\sigma^2=\inf\Bigl\{\int_{\mathbb T}|P|^2\,d\mu\,:\,P(z)=\sum_{k=0}^\ell a_kz^k,
\ a_k\in\mathbb C, \sum_{k=0}^\ell |a_k|^2=1\Bigr\}\,,
$$
and the standard density bound immediately yields the inequality
$$
\mathcal P\{0\le f(0),\dots, f(\ell)<\varepsilon\}\le \frac{\varepsilon^{\ell+1}}{\sqrt{\operatorname{det} A}}
\le \left[\frac\varepsilon\sigma\right]^{\ell+1}\,.
$$
Next we rebuild the minimizing polynomial $P$ into a polynomial $Q$ of degree $C\ell$ with positive
coefficients summing up to $1$ so that $|Q|\le e^{C\ell}|P|$ on $\mathbb T\setminus\mathcal L_\delta$.
Thus we get $\int_{\mathbb T}|Q|^2\,d\mu \le e^{2C\ell}\sigma^2$. Multiplying $Q$ by our universal 
polynomial $P_{L-C\ell,\text{universal}}$, we 
obtain a polynomial $\widetilde P$ of degree $L$ with $\ell+1$ consecutive 
big coefficients in the
middle satisfying $\int_{\mathbb T}|\widetilde P|^2\,d\mu \le e^{-2cL}\sigma^2$, provided that 
the ratio $\ell/L$ is small enough.

Finally, we run a chain of inequalities similar to (\ref{naive}) with the cutoff level 
$e^{-\frac c2L}\sigma$ instead of $e^{-\frac c2L}$ for $g$ and get the second term bounded by
$\left(\frac{Ce^{-\frac c2L}\sigma}{\sigma}\right)^{\ell+1}\le e^{-c'L^2}$.

This ``cancellation of $\sigma$ miracle'' is the only real novelty we introduce into the standard scheme
for the discrete case. The continuous case can be reduced to the discrete one with $\delta=\frac 14$ 
via dyadic spectral
decomposition, appropriate conditioning, and some other fairly routine trickery. For technical reasons we 
prefer to aim directly at the continuous case below, so the exact choices of parameters
we shall make and the notation we shall use for them will be slightly different from those in this
section. However, we hope that this short outline will help the reader to see what is going on when 
wading through various tangled details of the argument.

\section{The main lemma}

The following lemma is the basis for all our further considerations and may be of independent
interest as well.

\begin{lemma}
There exist $n_0\in\mathbb N$ and $c,c'>0$ such that if $f$ is a continuous stationary Gaussian process on 
$\R$ whose spectral measure $\mu$ is supported on $[-\frac 12,\frac 12]\setminus[-\frac 14,\frac 14]$, then for
every $n\ge n_0$, there exist a number $\sigma\in [0,\sqrt{\mu(\R)}]$ and a non-negative measure 
$\nu=\sum_{k=0}^n \beta_k\delta_k$ depending on $\mu$ and $n$ such that
$\nu(\R)=1$, 
$$
\mathcal E\left[\Bigl(\int_{\R}f\,d\nu\Bigr)^2\right]\le e^{-6cn}\sigma^2\,,
$$
and for every deterministic function $\varphi:\R\to\R$, we have
$$
\mathcal P\Bigl\{f+\varphi\ge 0\text{ on }\operatorname{supp}\nu, 
\int_{\R}(f+\varphi)\,d\nu\le e^{-cn}\sigma\Bigr\}
\le e^{-c'n^2}\,.
$$
Here, as usual, $\delta_k$ stands for the Dirac unit point mass at $k$.
\label{mainlem}
\end{lemma}

\begin{proof}
Let $n\in \mathbb N$ be sufficiently large. Fix $N\in[1,n]$ to be chosen later (the reader
should think of $N$ as of a small constant multiple of $n$) and consider the minimization
problem 
$$
\int_{\R}|P(e^{2\pi iy})|^2\,d\mu(y)\to\min, 
\qquad P(z)=\sum_{k=0}^N a_kz^k,\quad a_k\in\mathbb C, \quad \sum_{k=0}^N |a_k|^2=1\,.
$$
Let $P$ be a minimizing polynomial and let $\sigma^2$ be the value of the minimum (since $P(z)=1$ is an
admissible polynomial, we certainly have $\sigma^2\le\mu(\R)$). The polynomial $P(z)$ can be factored as
$$
P(z)=a\prod_{k=1}^N L_k(z)
$$
where $a\in\mathbb C$ and each $L_k(z)$ is a linear polynomial corresponding to a single root of $P$.  We
will put $L_k(z)=z-z_k$ for the roots of $P$ inside the unit disk and 
$L_k(z)=1-z_k z$ for the roots of $P$ outside the unit disk.
Note that this way we always have $|z_k|\le 1$ and, thereby, $|L_k(z)|\le 2$ for $|z|=1$; so
$$
1=\int_{-\frac 12}^{\frac 12}|P(e^{2\pi iy})|^2\,dy\le 2^{2N}|a|^2\,,
$$
whence $|a|\ge 2^{-N}$.

We will now replace each factor $L_k(z)$ by a polynomial $\wt L_k(z)$ of low degree with positive 
coefficients summing to $1$ (or, equivalently, satisfying $\wt L_k(1)=1$) 
so that $|\wt L_k(z)|\le 3|L_k(z)|$ for all $z$ on the left unit semicircle
$\mathbb T_-=\{z:|z|=1,\operatorname{Re}z\le 0\}$\,. Consider two cases

\leftline{\textbf{Case 1:} $\operatorname{dist}(z_k,\mathbb T_-)\ge \frac 12$.} 

In this case we just put $\wt L_k(z)=1$. Then on $\mathbb T_-$, we have $|\wt L_k(z)|\le 2|L_k(z)|$. 

\leftline{\textbf{Case 2:} $\operatorname{dist}(z_k,\mathbb T_-)< \frac 12$.} 
In this case the absolute value of the argument of $z_k$ is at least $\frac\pi 3$ and, therefore, $0$ 
is in the convex hull of $1,z_k,z_k^2,z_k^3$, so there exist $\alpha_0,\dots,\alpha_3\ge 0$ with 
$\alpha_0+\dots+\alpha_3=1$ such that $\alpha_0+\alpha_1z_k+\alpha_2z_k^2+\alpha_3z_k^3=0$.
Let $U(z)=\alpha_0+\alpha_1z+\alpha_2z^2+\alpha_3z^3$. Notice that $U(z_k)=0$ and $|U'(z)|\le 3$
in the unit disk, so $|U(\zeta)|\le 3|\zeta-z_k|$ if $|\zeta|=1$.

If $L_k(z)=z-z_k$, put 
$\wt L_k(z)=U(z)$. Then
$|\wt L_k(z)|=|U(z)|\le 3|z-z_k|=3|L_k(z)|$ on $\mathbb T_-$.

If $L_k(z)=1-z_kz$, put $\wt L_k(z)=z^3U(1/z)$. Then
$|\wt L_k(z)|=|U(1/z)|\le 3\left|\frac 1z-z_k\right|=3|L_k(z)|$ on $\mathbb T_-$.

Now put $\wt P(z)=\prod_{k=1}^N \wt L_k(z)$. Note that $\wt P$ is a polynomial of degree at most $3N$ 
with positive coefficients summing to $1$. Moreover, on $\mathbb T_-$, one has 
$$
|\wt P|=\prod_{k=1}^N|\wt L_k|\le 3^N \prod_{k=1}^N|L_k|\le\frac{3^N}{|a|}|P|\le 6^N|P|\,.
$$

Let $m\in \mathbb N$. Consider the polynomial 
$$
Q(z)=\frac{1}{m+4N+1}(1+z+\dots+z^{m+4N})\left(\frac{1+z}2\right)^m\wt P(z)=\sum_{k\ge 0}\beta_k z^k\,.
$$
This polynomial still has non-negative coefficients summing up to $1$ but, 
since $|1+z|\le\sqrt 2$ on $\mathbb T_-$, it satisfies the bound
$$
|Q(z)|\le 6^N 2^{-m/2}|P(z)|,\quad z\in\mathbb T_-\,.
$$
Next, the degree of $Q$ is at most $2m+7N$. 
At last, notice that the coefficients of $Q$ can be obtained by convolving the coefficients of 
the polynomial $\left(\frac{1+z}2\right)^m\wt P(z)$ of degree at most $m+3N$ with the coefficients 
of $\frac{1}{m+4N+1}(1+z+\dots+z^{m+4N})$, which form a flat sequence of length $m+4N+1$.  
Thus the coefficients of $Q$ with indices from 
$m+3N$ to $m+4N$ are equal to $\frac{1}{m+4N+1}$ (the common value of the coefficients
of $\frac{1}{m+4N+1}(1+z+\dots+z^{m+4N})$, which is $\frac{1}{m+4N+1}$, times 
the {\em full} sum of the coefficients of $\left(\frac{1+z}2\right)^m\wt P(z)$, which is $1$).

Choosing $m=8N$ and $N=\left\lfloor \frac n{23}\right\rfloor$, say, and putting $\nu=\sum_{k\ge 0}\beta_k\delta_k$,
(here, as above, $\beta_k$ are the coefficients of the polynomial $Q$) we see that $\nu$ is supported on $\{0,1,\dots,n\}$ and 
$$
%\begin{multline*}
\mathcal E\left[\Bigl(\int_{\R}f\,d\nu\Bigr)^2\right]
=\int_{\R}|Q(e^{2\pi iy})|^2\,d\mu(y)\le 2^{-2N}\int_{\R}|P(e^{2\pi iy})|^2\,d\mu(y)
%\\
=2^{-2N}\sigma^2\le e^{-6cn}\sigma^2
%\end{multline*}
$$
with $c=\frac{\log 2}{100}$, say, provided that $n$ is not too small. Let us now fix this
value of $c$ and estimate the probability
$
\mathcal P\Bigl\{f+\varphi\ge 0\text{ on }\operatorname{supp}\nu, 
\int_{\R}(f+\varphi)\,d\nu\le e^{-cn}\sigma\Bigr\}
$. 

Notice that on the event in question, for $k=m+3N,\dots,m+4N$, we must have 
$$
0\le f(k)+\varphi(k)\le \frac{1}{\beta_k}\int_{\mathbb R}(f+\varphi)\,d\nu
\le (m+4N+1)e^{-cn}\sigma\le ne^{-cn}\sigma\,.
$$
If $\sigma=0$, then this implies that the probability we are interested in is just $0$.
Otherwise, let $A$ be the covariance matrix of the vector $F=(f(m+3N),\dots,f(m+4N))$. By the stationarity of $f$,
it is the same as the covariance matrix of the vector $(f(0),\dots,f(N))$. It follows immediately 
from the definition of the spectral measure and the construction of $\sigma$ that the least eigenvalue
of $A$ is $\sigma^2$. Thus, the density of the distribution of the Gaussian vector $F$ in $\R^{N+1}$
 is bounded by
$$
(2\pi)^{-\frac{N+1}2}(\det A)^{-\frac 12}\le 1\cdot\sigma^{-N-1}=\sigma^{-N-1}\,.
$$
On the other hand, on the event under consideration, $F$ belongs to a cube with sidelength $ne^{-cn}\sigma$ 
whose Euclidean volume is $(ne^{-cn}\sigma)^{N+1}$. Hence, the probability in question is at most 
$(ne^{-cn})^{N+1}\le e^{-c'n^2}$, provided that $n$ is not too small. The lemma is completely proved. 
\end{proof}

Let us make two remarks:

\begin{remark}\rm
Considering the process $f(x/a)$ instead of $f(x)$ with some $a>0$, we can immediately generalize 
this result to the case when the spectral measure $\mu$ is supported on $[-\frac a2,\frac a2]\setminus[-\frac a4,\frac a4]$. In this case the measure $\nu$ will be supported on the set $\{0,\frac 1a,\dots,\frac na\}$ but the rest of the formulation of Lemma \ref{mainlem} will remain exactly the same; in particular, the constants $n_0$, $c$ and $c'$ will not depend on $a$ in any way.
\label{rema}
\end{remark}

\begin{remark}\rm
The argument in the proof of Lemma \ref{mainlem} applies with some minor changes to the case when 
the spectral gap is $[-\delta,\delta]$ with some small $\delta>0$ instead of $[-\frac 14,\frac 14]$.
This allows one to almost immediately get the bound $e^{-c(\delta)n^2}$ with some $c(\delta)>0$
for the positivity probability in the discrete case. However, the dependence of $c(\delta)$ on 
$\delta$ one could get on this way would be suboptimal, so we will not use this most direct 
approach but, instead, will resort to a more elaborate scheme that would allow us to treat the
continuous case and spectral measures with unbounded supports as well. 
\label{rema2}
\end{remark}

\section{The dyadic decomposition of the spectral measure and the proof of Theorem \ref{spectralgap}}

Let $f$ be any continuous stationary Gaussian process with any spectral measure $\mu$ that decays not
too slowly at infinity and satisfies the gap condition $\mu([-\delta,\delta])=0$. Note that the positivity probability is never greater than $\frac 12$, which is the probability that $f$ is non-negative at a single
point, so it will suffice to prove inequality (\ref{main}) under the assumption that $\delta L$ is greater than
some fixed absolute constant.

We start with the decomposition $\mu=\sum_{a}\mu_a$ where $a$ runs over the numbers of the kind $2^k\delta$, $k\ge 2$, and $\mu_a$ is just the part of the measure $\mu$ supported on 
$[-\frac a2,\frac a2]\setminus[-\frac a4,\frac a4]$. This decomposition of the spectral measure corresponds to a decomposition of the continuous stationary Gaussian process $f$ into the sum of independent
processes $f_a$. For each $a$, fix an integer $n_a\ge n_0$ to be chosen later.

By Remark \ref{rema}, for each $a$, we can find a non-negative measure $\nu_a$ of total
mass $1$ supported on $\{0,\frac 1a,\dots,\frac{n_a}a\}$ and a number $\sigma_a\in[0,\sqrt{\mu_a(\R)}]$ such that
$$
\mathcal E\left[\Bigl(\int_{\R}f_a\,d\nu_a\Bigr)^2\right]\le e^{-6cn_a}\sigma_a^2\,,
$$
and for every deterministic function $\varphi_a:\R\to\R$, we have
$$
\mathcal P\Bigl\{f_a+\varphi_a\ge 0\text{ on }\operatorname{supp}\nu_a, 
\int_{\R}(f_a+\varphi_a)\,d\nu_a\le e^{-cn_a}\sigma_a\Bigr\}
\le e^{-c'n_a^2}\,.
$$
Since $f_a$ is stationary, the same inequalities hold for any shift of the measure $\nu_a$. Consider
now the (countably infinite in general) convolution $\nu=*_a\nu_a=\nu_{4\delta}*\nu_{8\delta}*\dots$. 
Since each measure $\nu_a$ is non-negative and satisfies $\nu_a(\mathbb R)=1$, this convolution is well-defined
and supported
on $[0,L]$, provided that
\begin{equation}
\sum_a\frac {n_a}{a}\le L\,.
\label{cond1}
\end{equation}
Using the Minkowski inequality, we get the bound
$$
\left[\mathcal E\Bigl(\int_{\R}f\,d\nu\Bigr)^2\right]^{\frac 12}
\le
\sum_a \left[\mathcal E\Bigl(\int_{\R}f_a\,d\nu\Bigr)^2\right]^{\frac 12}\,.
$$
Note now that we can write $\nu$ as 
$\nu_{a}*\nu^{(a)}$ where 
$\nu^{(a)}=*_{a'\ne a}\nu_{a'}$ is also a measure of total mass $1$. 
Then, denoting by $\nu_{a,t}$ the
shift of the measure $\nu_{a}$ by $t\in\mathbb R$ (so $\nu_{a,t}(E)=\nu_{a}(E-t)$),
and using the integral version of the Minkowski inequality, we get
\begin{multline*}
\left[\mathcal E\Bigl(\int_{\R}f_a\,d\nu\Bigr)^2\right]^{\frac 12}
=
\left[\mathcal E\Bigl(\int_{\R}\Bigl[\int_{\R}f_a\,d\nu_{a,t}\Bigr]\,d\nu^{(a)}(t)\Bigr)^2\right]^{\frac 12}
\\
\le
\int_{\R} \left[\mathcal E\Bigl(\int_{\R}f_a\,d\nu_{a,t}\Bigr)^2\right]^{\frac 12}\,d\nu^{(a)}(t)
=
\left[\mathcal E\Bigl(\int_{\R}f_a\,d\nu_{a}\Bigr)^2\right]^{\frac 12}
\le e^{-3cn_a}\sigma_a\,.
\end{multline*}
Here we used the fact that the Gaussian process $f_a$ is stationary, so the distribution of $\int_{\R}f_a\,d\nu_{a,t}$
does not depend on $t$.

Hence,
$$
\left[\mathcal E\Bigl(\int_{\R}f\,d\nu\Bigr)^2\right]^{\frac 12}
\le
\sum_a e^{-3cn_a}\sigma_a\le \max_a(e^{-2cn_a}\sigma_a)\,,
$$
provided that
\begin{equation}
\sum_a e^{-c n_a}\le 1\,.
\label{cond2}
\end{equation}
The maximum always exists because $\sigma_a$ tend to $0$ as $a\to\infty$ (recall that 
$\sigma_a^2\le\mu_a(\R)$ and $\sum_a\mu_a(\R)=\mu(\R)<+\infty$). We can also assume that
the maximum is strictly positive since, otherwise, the only chance for $f$ to be non-negative
on $\operatorname{supp}\nu$ is to be identically $0$ there and that event has zero probability.

Let $\al$ be the value of $a$ for which the maximum is attained. By the standard Gaussian tail estimate,
we have
$$
\mathcal P\left\{\int_{\R}f\,d\nu\ge\frac 12e^{-cn_{\al}}\sigma_{\al}\right\}\le 
\exp\left(-\frac {e^{2cn_{\al}}}8\right)\le e^{-\delta^2L^2}\,,
$$
provided that 
\begin{equation}
\min_a n_a\ge c''\delta L\text{ with some }c''>0
\label{cond3}
\end{equation}
and $\delta L$ is not too small.

Thus, it will suffice to bound the probability of the event $S$ that $f\ge 0$ on $[0,L]$ and
$\int_{\R}f\,d\nu\le \frac 12e^{-cn_{\al}}\sigma_{\al}$. 
%Let $\nu^{(\al)}=*_{a\ne \al}\nu_a$ so that 
%$\nu=\nu_{\al}*\nu^{(\al)}=\int_{\R}\nu_{\alpha,t}\,d\nu^{(\al)}(t)$ where $\nu_{\alpha,t}$ is the
%shift of the measure $\nu_{\al}$ by $t$, i.e., $\nu_{\alpha,t}(E)=\nu_{\al}(E-t)$.
For $t\in\R$, let $S_t$ be the event that $f\ge 0$ on $[0,L]$
and $\int_{\R}f\,d\nu_{\al,t}\le e^{-cn_{\al}}\sigma_{\al}$.
Writing $f=f_\al+\sum_{a\ne\al}f_a=f_\al+\varphi_\al$ and 
conditioning upon all $f_a$ with $a\ne\al$, we see that for every fixed $t\in\operatorname{supp}\nu^{(\al)}$,
the event $S_t$ is (conditionally) contained in the event that $f_{\al}+\varphi_{\al}\ge 0$ on 
$\operatorname{supp}\nu_{\al,t}$ and 
$\int_{\R}(f_{\al}+\varphi_{\al})\,d\nu_{\al,t}\le e^{-cn_{\al}}\sigma_{\al}$.
Thus, the probability of $S_t$ does not exceed $e^{-c'n_{\al}^2}$.

Now define $g(\omega)=\int_{\R}\chi_{S_t}(\om)\,d\nu^{(\al)}(t)$. Then, on the one hand,
$$
\mathcal E[g]\le \sup_{t\in\operatorname{supp}\nu^{(\al)}}\mathcal P\{S_t\}\le e^{-c'n_{\al}^2}\,.
$$
On the other hand, on $S$ we must have $g\ge \frac 12$ because otherwise 
$$
\nu^{(\al)}\left\{t:\int_{\R}(f_{\al}+\varphi_{\al})\,d\nu_{\al,t}> e^{-cn_{\al}}\sigma_{\al}\right\}>\frac 12$$ 
and
$$
\int_{\R}f\,d\nu= \int_{\R}\left[\int_{\R}(f_{\al}+\varphi_{\al})\,d\nu_{\al,t}\right]\,d\nu^{(\al)}(t)
>\frac 12e^{-cn_{\al}}\sigma_{\al}\,,
$$ 
which contradicts the second property in the definition of $S$.

This yields the bound
$$
\mathcal P\{S\}\le 2e^{-c'n_{\al}^2}\le e^{-c'''\delta^2L^2}\,
$$ 
provided that (\ref{cond3}) holds and $\delta L$ is not too small.

It remains to show that we can, indeed, choose $n_a$ satisfying (\ref{cond1}--\ref{cond3}).
We will just take a small $c''>0$ and, for $a=2^k\delta$, put $n_a=\lfloor c''2^{k/2}\delta L\rfloor$.
Then  
(\ref{cond1}) rewrites as
$$
\sum_{k\ge 2} c''2^{-k/2}\le 1\,,
$$
which can be ensured by an appropriate choice of $c''>0$,
while, 
(\ref{cond2}), (\ref{cond3}) and the inequality $\min_a n_a\ge n_0$ are satisfied as long as $\delta L$ is not too small. 

This finishes the proof of the desired bound in the continuous case. To handle the discrete case (stationary Gaussian processes on $\mathbb Z$ with spectral gap), it suffices to note that every such
 discrete process can be viewed as the restriction
of a continuous process with the spectral measure supported on $[-\frac 12,\frac 12]$ with the same spectral gap. 
If we assume 
that $\delta>0$ is a negative power of $2$ (which we can always do without loss of generality)
and restrict the dyadic decomposition in the 
argument above 
 to $a\le 1$, 
then all measures $\nu_a$ and their convolutions will be supported on $\mathbb Z$. 
 
\section{The sharpness of the bound}
\label{sharpness}

%\end{document} 
 
Now we will present a simple theorem that provides a wide class of spectral measures
for which the result of Theorem \ref{spectralgap} cannot be substantially improved.
\begin{theorem}
\label{lowerbound}
Let $f$ be any continuous stationary Gaussian process whose spectral measure $\mu$ is supported on
$[-R,R]$
and satisfies 
\begin{equation}
\rho_n^2=\inf\left\{\int_{\R}|P|^2\,d\mu\,:\,P(y)=1+\sum_{k=1}^n a_k y^k, a_k\in\mathbb C\right\}
\ge e^{-2Cn}\mu(\R)
\label{rho}
\end{equation}
for all $n\ge 1$ with some $C>0$. 
Then, for some $C'>0$, $L_0>0$ depending on $C$ only, we have 
$$
\mathcal P\{f\ge 0 \text{ on }[0,L]\}\ge e^{-C'R^2L^2}
$$
for all $L\ge L_0$.
\end{theorem}

Note that this requirement is compatible with the spectral gap condition, i.e., there exist compactly
supported measures $\mu$ with zero mass on some interval around the origin such that
(\ref{rho}) holds. For instance,
the classical Remez theorem (see \cite{CW}, Lemma $4$) shows that the inequality 
(\ref{rho}) is satisfied for every measure $\mu$ of total mass $1$ that 
has a non-trivial absolutely continuous component. This observation immediately yields the following 
\begin{corollary}
Let $f$ be any continuous stationary Gaussian process whose spectral measure $\mu$ is compactly
supported and has a non-trivial absolutely continuous component.
Then, for some $C'>0$, $L_0>0$, we have 
$$
\mathcal P\{f\ge 0 \text{ on }[0,L]\}\ge e^{-C'L^2}
$$
for all $L\ge L_0$.
\end{corollary}
In general, the compact support condition in the Corollary cannot be removed if one
wants to preserve the conclusion in the current form
(see \cite{FFN}, Corollary 5, for an example of a continuous stationary Gaussian process 
whose spectral measure is absolutely continuous and for which
$\mathcal P\{f\ge 0 \text{ on }[0,L]\}\le e^{-e^{cL}}$ for large $L$).
However, having a non-trivial absolutely continuous part is by no means necessary for the conditions
of Theorem \ref{lowerbound} to be satisfied. At the end of this section we will present a purely
discrete measure $\mu$ that has a spectral gap around the origin and still satisfies (\ref{rho}).

\begin{proof}
Since $\mu$ is compactly supported, the Gaussian process $f$ 
represents a random entire function of exponential type and we have 
the Taylor series decomposition
$$
f(x)=\sum_{k\ge 0} \frac 1{k!}f^{(k)}(0)x^k
$$
converging almost surely on the entire real line.

Fix some $L>0$. 
We will certainly have $f$ preserving sign on $[0,L]$ 
if we can find some positive numbers $a,a_1,a_2,\dots$ such that
$a\ge \sum_{k\ge 1}a_k\frac{L^k}{k!}$, $|f(0)|> a$, and $|f^{(k)}(0)|\le a_k$ for all $k\ge 1$. 
Since $f$ and $-f$ are equidistributed, we have $f\ge 0$ on one half (with respect to the probability
measure) of that event.
In what follows, we  assume without loss of generality that $R=\frac 1{2\pi}$, $\mu(\R)=1$.

When choosing $a$ and $a_k$, it will be convenient to take care of ``small'' and ``large'' $k$ separately.
So, fix some big $K\ge 1$ (it will end up being a constant multiple of $L$ for large $L$) and
consider first the derivatives $f^{(k)}(0)$ with $k\le K$. 

Condition (\ref{rho}) shows, in particular, that
the Gaussian random variable $f(0)$ cannot be completely determined by the vector $F=(f^{(1)}(0),\dots,f^{(K)}(0))$.
More precisely, we can represent $f(0)$ as $\rho_K g+AF$ where $g$ is the standard real Gaussian random variable
independent of $F$ and $A$ is some linear mapping from $\R^K$ to $\R$. 

Indeed, let $A=(a_1,\dots,a_K)\in \mathbb C^K$ 
be a minimizer of
$$
\mathcal E[|f(0)-AF|^2]=\int_{\mathbb R}\Bigl|1-\sum_{k=1}^K (2\pi i)^k a_k y^k\Bigr|^2\,d\mu(y)\,.
$$  
Note that the expression on the right hand side of this identity shows that the minimization problem we are talking about
is just the problem of finding the closest point to the constant function $1$ in the finite-dimensional subspace
of the complex Hilbert space $L^2(\mu)$ spanned by $y^k$ ($k=1,\dots,K$), so it always has a solution, which
is just the orthogonal projection of $1$ to that subspace. 

On the other hand, the expression on left hand side shows 
that the minimizing vector $A$ can always be taken real (since $f(0),f^{(k)}(0)$ are all real Gaussian random
variables, removing the imaginary part does not increase the functional), and,  
for this choice of a minimizing vector, $f(0)-AF$ is a real Gaussian random variable satisfying 
$\mathcal E[(f(0)-AF)f^{(k)}(0)]=0$ for all $k=1,\dots,K$. By the special properties of jointly
Gaussian random variables, it means that $f(0)-AF$ is independent of $F$. At last, by the definition 
of $\rho_K$ (see (\ref{rho})), we get 
$$
\mathcal E[(f(0)-AF)^2]=\rho_K^2
$$
so we can write $f(0)-AF$ as $\rho_K$ times a standard real Gaussian $g$.

Since we want to impose the restriction that $F$ is small, our only chance to get $f(0)$ reasonably large 
is to use the $\rho_K g$ component of this decomposition, which dictates the choice $a=\rho_K\ge e^{-CK}$.
We will also put $a_k=\alpha>0$ for $k=1,\dots, K$. Since we need to ensure that
$a\ge \sum_{k\ge 1}a_k\frac{L^k}{k!}$ in the end, we choose $\alpha=\frac a2e^{-L}$,
so that
$$
\sum_{k=1}^K\frac{L^k}{k!}a_k< e^L\al=  \frac a2\,.
$$

Our next goal will be to estimate from below the probability that $|f(0)|>a$ while $|f^{(k)}(0)|<\al$ for all $k=1,\dots,K$.
Since 
$$
\mathcal{E}[f^{(k)}(0)^2]=(2\pi)^{2k}\int_{\R} y^{2k}\,d\mu(y)\le 1\,,
$$
for all $k$, the diagonal elements of the covariance matrix of $F$ are
bounded by $1$. Therefore its norm is at most $K$, whence $F$ can be written as $BG$ where $B:\R^K\to\R^K$ is a linear transformation of norm at most $\sqrt K$ and $G=(g_1,\dots,g_K)$ is the standard Gaussian vector in $\R^K$. 

Thus, denoting the Euclidean norm of the vector $F$ by $|F|$, as usual, and observing that $\alpha<1$, we have 
\begin{multline*}
\mathcal P\{|f^{(k)}(0)|\le \alpha\text{ for all }k=1,\dots,K\}
\\
\ge
\mathcal P\{|F|\le \al\}\ge \mathcal P\{|G|<K^{-\frac 12}\al\}\ge 
\mathcal P\{|g_k|<K^{-1}\al\text{ for all }k=1,\dots,K\}
\\
\ge 
\left(\frac{2\al e^{-\frac 12K^{-2}\al^2}}{\sqrt{2\pi}\,K}\right)^K 
\ge \left(\frac{\al}{3K}\right)^K\,, 
\end{multline*}
say.

Conditioning upon $F$, we see that
for every $\bar F\in\R^K$,
$$
%\mathcal P\{|f(0)|>a\mid F=\bar F\}=
\mathcal P\{|f(0)|>\rho_K\mid F=\bar F\}
\ge\inf_{t\in\R}\mathcal P\{|\rho_K g+t|>\rho_K\}=\mathcal P\{|g|>1\}=p_0>0
$$
(here we used the fact that for every even unimodal probability density function $\gamma$, the integral
of $\gamma$ over an interval of fixed length is maximized by the interval centered at the origin).
Therefore
\begin{multline*}
\mathcal P\{|f(0)|>a;\ |f^{(k)}(0)|\le a_k\text{ for all }k=1,\dots,K\}
\\
=\mathcal P\{|f(0)|>\rho_K;\ |f^{(k)}(0)|\le \alpha \text{ for all }k=1,\dots,K\}
\\
\ge 
p_0\left(\frac{\al}{3K}\right)^K\,.
\end{multline*}
Plugging in the value $\alpha=\frac a2e^{-L}=\frac 12\rho_K e^{-L}\ge \frac 12e^{-CK-L}$ (recall that 
by (\ref{rho}) we have $\rho_K\ge e^{-CK}$), we 
get the lower bound
$$
p_0\left(\frac{\al}{3K}\right)^K\ge p_0\left(\frac{e^{-CK-L}}{6K}\right)^K\,.
$$
For $K\ge L$, the right hand side of the last inequality is at least $e^{-\wt CK^2}$
with $\wt C=C+7+p_0^{-1}$, say.

Now let us take care of $k> K$.
We have already seen that 
$$
\mathcal{E}[f^{(k)}(0)^2]\le 1\,,
$$
so we have the standard Gaussian tail bound
$$
\mathcal P\{|f^{(k)}(0)|>a_k\}\le e^{-\frac12 a_k^2}\,.
$$
Thus, for any choice of $a_k>0$, the probability that the inequality $|f^{(k)}(0)|\le a_k$ is violated for some $k>K$
is at most $\sum_{k>K}e^{-\frac 12a_k^2}$ and, therefore,
\begin{multline*}
\mathcal P\{|f(0)|>a;\ |f^{(k)}(0)|\le a_k\text{ for all }k\ge 1\}
\\
=
\mathcal P\{|f(0)|>a;\ |f^{(k)}(0)|\le a_k\text{ for all }k=1,\dots,K;\text{ and }
|f^{(k)}(0)|\le a_k\text{ for }k>K\}
\\
\ge 
e^{-\wt CK^2}-\sum_{k>K}e^{-\frac 12 a_k^2}\,.
\end{multline*}

We want to make sure that $\sum_{k>K}e^{-\frac 12a_k^2}$ stays
well below $e^{-\wt CK^2}$ so that the subtraction in the probability estimate is harmless. 
It can be achieved, say, by putting $a_k=\sqrt{2(\wt CK^2+k)}$ for $k>K$, in which case the sum in question is bounded by
$e^{-\wt CK^2}\sum_{k>K}e^{-k}\le \frac 12e^{-\wt CK^2}$, so we still have the lower bound   
$\frac 12e^{-\wt CK^2}$ for the probability of the event we are interested in. 

Since we have already ensured that $\sum_{1\le k\le K}\frac{L^k}{k!}a_k<\frac a2$, it remains to choose $K$ 
so that $\sum_{k>K}\frac{L^k}{k!}a_k\le \frac a2$ as well. Recalling that
$a=\rho_K\ge e^{-CK}$, we see that it will suffice to ensure that
$$
\sum_{k>K}\frac{L^k}{k!}\sqrt{2(\wt CK^2+k)}\le \frac 12 e^{-CK}\,.
$$
The classical bound $k!\ge \left(\frac ke\right)^k$  and the inequality
$$
\sqrt{2(\wt CK^2+k)}\le \sqrt{2(\wt Ck^2+k)}\le 2\wt Ck\le e^{2\wt C k}
$$
imply that the left hand side is at most 
$$
\sum_{k>K}\left(\frac{Le^{2\wt C+1}}{K}\right)^k\,.
$$
If $K\ge 2e^{2\wt C+1}L$ (i.e., the common ratio of this geometric progression is at most $\frac 12$), then the
sum converges and does not exceed the term of the progression corresponding to $k=K$, which is
$\left(\frac{Le^{2\wt C+1}}{K}\right)^K$.
To ensure that it is less than $\frac 12e^{-CK}>\left(e^{-C-1}\right)^K$, it is enough 
to choose $K$ such that
$$
\frac{Le^{2\wt C+1}}{K}\le e^{-C-1}\,,
$$
i.e.,
$
K\ge e^{C+2\wt C+2}L\,.
$
Compared to all the previous conditions imposed on $K$ ($K\ge 1$, $K\ge L$, $K\ge 2e^{2\wt C+1}L$),  this one is the 
most restrictive for $L\ge 1$ but it still allows one to choose $K$ below or at some fixed constant multiple of $L\ge1$, so the desired estimate follows.
\end{proof}

Now, it remains to present a symmetric discrete measure
$\mu$ with a spectral gap that satisfies the assumptions of Theorem \ref{lowerbound}. We 
will merely take for $\mu$ the measure whose restriction to $[0,+\infty)$ is 
$\sum_{n\ge 2}\frac{1}{n2^n}\sum_{k=1}^n\delta_{\frac {n+k}{4\pi n}}$.
It is supported on $[-\frac 1{2\pi}\frac 1{2\pi}]\setminus[-\frac 1{4\pi},\frac 1{4\pi}]$
and satisfies (\ref{rho}) just by the Lagrange interpolation formula
\begin{multline*}
1=P(0)=\sum_{k=1}^{n+1}\Bigl[\prod_{1\le j\le n+1, j\ne k}\frac{0-\frac{n+1+j}{4\pi(n+1)}}{\frac{k-j}{4\pi(n+1)}}\Bigr]
P\left(\frac{n+1+k}{4\pi(n+1)}\right)
\\
\le \sum_{k=1}^{n+1}\frac {(2n+2)!}{(k-1)!(n+1-k)!(n+1)!}\left|P\left(\frac{n+1+k}{4\pi(n+1)}\right)\right|
\\
\le 3^{2n+1}(2n+2)\sum_{k=1}^{n+1}\left|P\left(\frac{n+1+k}{4\pi(n+1)}\right)\right|
\\
\le 3^{2n+1}2^{n+2}(n+1)^2\int_{\R}|P|\,d\mu\le 10^{3n}\sqrt{\int_{\R}|P|^2\,d\mu}\,.
\end{multline*}

\end{document}